\documentclass[11pt]{article}
\usepackage{amsmath,amssymb,amsthm,eucal}
\usepackage[usenames,dvipsnames]{color}
\usepackage{amscd}
\usepackage[all]{xy}
\usepackage[colorlinks=true,linkcolor=black,citecolor=black]{hyperref}

\title{\vspace*{-1pc}%
       The Cuntz-Pimsner extension and mapping cone exact sequences}

\author{
Francesca Arici\dag, Adam Rennie\ddag\\[3pt]
\dag Institute for Mathematics, Astrophysics and Particle Physics, FNWI,\\ Radboud University Nijmegen, Heyendaalseweg 135, 6525AJ Nijmegen, The Netherlands
\\[3pt]
\ddag School of Mathematics and Applied Statistics,
University of Wollongong,\\
Northfields Ave 2522, Australia}

\advance\topmargin by -\headheight
\advance\topmargin by -\headsep
\evensidemargin=\oddsidemargin
\makeatletter
\def\section{\@startsection{section}{1}{\z@}{-3.5ex plus -1ex minus
  -.2ex}{2.3ex plus .2ex}{\large\bf}}
\def\subsection{\@startsection{subsection}{2}{\z@}{-3.25ex plus -1ex
  minus -.2ex}{1.5ex plus .2ex}{\normalsize\bf}}
\makeatother

\numberwithin{equation}{section} 

\theoremstyle{plain} 
\newtheorem{thm}{Theorem}[section]
\newtheorem{lemma}[thm]{Lemma}
\newtheorem{prop}[thm]{Proposition}

\newtheorem{ass}{Assumption}
\newtheorem{ex}[thm]{Example}

\newtheoremstyle{dotless}{}{}{\itshape}{}{\bfseries}{}{ }{}
\theoremstyle{dotless}
\newtheorem*{thm*}{Theorem}

\theoremstyle{definition} 

\theoremstyle{remark} 
\newtheorem{rmk}[thm]{Remark}

\DeclareMathOperator{\End}{End}   
\DeclareMathOperator{\Id}{Id}     

\newcommand{\phimod}{\Xi_A} 
\newcommand{\Fock}{\F_E} 
\newcommand{\algFock}{F_E^{\textnormal{alg}}} 
\newcommand{\Pim}{O_E} 
\newcommand{\Toe}{\mathcal{T}_{E}} 
\newcommand{\core}{\Pim^\gamma} 

\newcommand{\al}{\alpha}      

\newcommand{\C}{\mathbb{C}}   
\newcommand{\D}{\mathcal{D}}  
\newcommand{\F}{\mathcal{F}}  
\renewcommand{\H}{\mathcal{H}}  
\newcommand{\K}{\mathrm{End}^0_A}

\newcommand{\N}{\mathbb{N}}   
\newcommand{\ox}{\otimes}     
\newcommand{\hox}{\hat\otimes} 
\renewcommand{\O}{\mathcal{O}}  
\newcommand{\R}{\mathbb{R}}   
\renewcommand{\S}{\mathcal{S}} 
\newcommand{\T}{\mathcal{T}} 
\newcommand{\Z}{\mathbb{Z}}   
\renewcommand{\tilde}{\widetilde}

\newcommand{\Mult}{\textup{Mult}}
\newcommand{\Cs}{$C^*$-}
\newcommand{\smeb}{self-Morita equivalence bimodule}
\newcommand{\ev}{\textup{ev}}

\newcommand{\bott}{\textup{Bott}}
\newcommand{\BndE}{\mathrm{End}^*_A(E)}

\newcommand{\ei}{\mathrm{e}}

\newcommand{\stroke}{\mathbin|}     

\def\pairL_#1(#2|#3){{}_{#1}(#2\stroke#3)} 
\def\pairR(#1|#2)_#3{(#1\stroke#2)_{#3}} 
\def\scal<#1|#2>{\langle#1\stroke#2\rangle} 

\newcommand{\topop}{\mathfrak{q}} 

\newbox\ncintdbox \newbox\ncinttbox 
	\setbox0=\hbox{$-$}
	\setbox2=\hbox{$\displaystyle\int$}
	\setbox\ncintdbox=\hbox{\rlap{\hbox
		to \wd2{\hskip-.125em \box2\relax\hfil}}\box0\kern.1em}
	\setbox0=\hbox{$\vcenter{\hrule width 4pt}$}
	\setbox2=\hbox{$\textstyle\int$}
	\setbox\ncinttbox=\hbox{\rlap{\hbox
		to \wd2{\hskip-.175em \box2\relax\hfil}}\box0\kern.1em}

\begin{document}

\maketitle

\vspace{-2pc}

\begin{abstract}
For Cuntz-Pimsner algebras of bi-Hilbertian bimodules of finite Jones-Watatani index
satisfying some side conditions,
we give an explicit isomorphism between the $K$-theory exact sequences of the mapping cone
of the inclusion of the coefficient algebra into a Cuntz-Pimsner algebra, and the Cuntz-Pimsner
exact sequence. In the process we extend some results by the second author and collaborators from finite projective bimodules to finite index bimodules,
and also clarify some aspects of Pimsner's `extension of scalars' construction. 
\end{abstract}




\parskip=6pt
\parindent=0pt


\section{Introduction}
\label{sec:intro}
Mapping cones play an important role in studying the properties of 
$KK$-theory, \cite{CuSk86,MNest}, and have likewise been used to further
the study of non-commutative topology and dynamics \cite{CPR1,Putnam}.
The aim of this note is to make explicit, in a specific case, the abstract relationship between 
extensions of $C^*$-algebras and mapping cone extensions. 

Motivated by direct calculations with mapping cone and Cuntz-Pimsner exact sequences,
as in \cite{ABL,Arici}, we investigate the relationship between the defining extension  of
the Cuntz-Pimsner algebra $O_E$
\begin{equation}
\tag{\ref{eq:ext}}
\xymatrix{0 \ar[r] & \K(F_{E}) \ar[r] & \Toe \ar[r]^{\pi} & \Pim \ar[r] & 0,}
\end{equation}
(here $F_E$ is the Fock module, $\T_E$ the Toeplitz-Pimsner algebra and $\K$ denotes the algebra of compact endomorphisms)
 and the exact sequence of the mapping cone $M(A,O_E)$ of the inclusion of the 
 coefficient algebra $A$ into $O_E$ 
 \[
 \xymatrix{0 \ar[r] &\S \Pim \ar[r]& M(A,\Pim) \ar[r]& A \ar[r]& 0},
 \]
 where $\S \Pim$ is the suspension.
We show that we can construct an explicit isomorphism of
the associated $K$-theory sequences at the level of unbounded $KK$-cycles.
 
Abstractly, the existence of such an isomorphism follows from the fact that the $KK$-category 
is a triangulated category whose exact triangles are mapping cone triangles, with isomorphisms 
given by $KK$-equivalence (cf. \cite{MNest}). Indeed, 
for every semi-split extension with quotient map 
$\pi$, by \cite{CuSk86}, one has an isomorphism 
of triangles making the extension triangle equivalent 
to the mapping cone triangle of $\pi$, i.e. one has a commutative diagram of triangles 
where all ``vertical'' arrows are $KK$-equivalences.

%
More specifically, in the case of Cuntz-Pimsner algebras, 
when the coefficient algebra $A$ is nuclear, the defining extension 
is semi-split, and hence one obtains an 
isomorphism of the extension triangle with the mapping cone triangle for $\pi$
\begin{equation}
\label{eq:triang-pi}
 \xymatrix{\S\Pim \ar[r] &M(\Toe, \Pim) \ar[r]& \T_E \ar[r]^{\pi}& O_E }.
\end{equation}
Using the $KK$-equivalence between $A$ and $\Toe$ and the natural 
Morita equivalence between $A$ and $\K(\F_E)$, one can show that the mapping cone triangle
\begin{equation*}
\xymatrix{\S O_E \ar[r]& M(A,O_E)\ar[r]& A\ar[r]& \O_E}
\end{equation*}
for the inclusion  of the coefficient algebra $A$ into $O_E$
is in turn isomorphic to \eqref{eq:triang-pi}. This follows 
from the axioms of a triangulated category which imply that the 
mapping cone of $A\to O_E$ is unique up to a (non-canonical) isomorphism in $KK$. 
Combining the two isomorphisms of triangles, one obtains the isomorphism of exact triangles 
\[
\xymatrix{\S \Pim \ar[r]\ar[d]^{=}& M(A,\Pim) \ar[d]\ar[r]& A \ar[r]\ar[d]^{\alpha}& \Pim \ar[d]^{=}\\
\S \Pim \ar[r]& \K(\F_E) \ar[r]& \Toe \ar[r]& \Pim }
\]
which induces an isomorphism of the corresponding $KK$-exact sequences.  

In this paper we provide the isomorphism between the associated six-term exact sequences 
explicitly at the level of unbounded $KK$-cycles. This allows one to 
exploit these mapping cones in concrete computations. We indicate how this works in
the case of $C^*$-algebras of non-singular graphs.

Many of the constructions we rely on from \cite{kappa,RRS} were
proved for finitely generated bimodules over unital algebras. In order to deal with
suspensions we extend these results to handle the more general case of
bimodules with finite right Watatani index. 
Our main result is as follows.

\begin{thm*}{\bf \ref{thm:KK-equivalence}}
Let $E$ be a bi-Hilbertian $A$-bimodule of finite right Watatani index, full as a right module
with injective left action, and satisfying Assumptions \ref{ass:one} and \ref{ass:two} 
on pages 8 and 9 respectively. Let
$(O_E,\phimod,\D)$ be the unbounded representative of the defining extension of $O_E$,
and $(M(A,O_E),\hat\Xi_A,\hat\D)$ the lift to the mapping cone. Then
$$
\cdot\ox_{M(A,O_E)}[(M(A,O_E),\hat\Xi_A,\hat\D)]:\,K_*(M(A,O_E))\to K_*(A)
$$
is an isomorphism that makes diagrams in $K$-theory commute. If furthermore the 
algebra $A$ belongs to the bootstrap class, the Kasparov product with the class
$[(M(A,O_E),\hat\Xi_A,\hat\D)] \in KK(M(A,O_E),A)$ is a $KK$-equivalence.
%
%
Together with the identity map, 
$\cdot\ox_{M(A,O_E)}[(M(A,O_E),\hat\Xi_A,\hat\D)]$ induces
an isomorphism of $KK$-theory exact sequences. 
\end{thm*}

{\bf Acknowledgements.} Both authors thank Magnus Goffeng,
Jens Kaad, Bram Mesland, Ryszard Nest and Aidan Sims for discussions
regarding the change of scalars argument (JK and BM) 
and other fruitful discussions. We also thank
Aidan Sims for comments on an earlier version of this article.
FA thanks Georges Skandalis for his hospitality in Paris and for useful 
comments. FA was partially supported by the GNSAGA of INdAM 
and by NWO under the VIDI-grant 016.133.326. AR acknowledges the support
of the Australian Research Council.

\section{Finite index bi-Hilbertian bimodules for non-unital algebras}
\label{sec:big-words}

We start by recalling the basic setup of \cite{RRS} and \cite{kappa}, and show how 
it extends to non-unital algebras using more refined constructions from \cite{KajPinWat}. 
In \cite{RRS} and \cite{kappa}, the basic data 
was a unital separable nuclear $C^*$-algebra $A$, and a
bi-Hilbertian bimodule $E$ over $A$ in the sense 
of \cite[Definition~2.3]{KajPinWat}, which is finitely generated and projective for 
both the right and left module structures. 

In this paper we will dispense with the unitality of the algebra $A$, and consequently
also the finitely generated and projective hypotheses on the module $E$. So
we will assume throughout the paper that $E$ is a countably generated bimodule over
$A$, which carries both left and right $A$-valued inner products ${}_A(\cdot|\cdot)$, $(\cdot|\cdot)_A$ 
for which the respective actions are injective and adjointable, and $E$ is complete. 
The two inner products automatically yield equivalent norms (see, for instance \cite[Lemma 2.2]{RRS}).
We write ${}_AE$ for $E$ when we wish to emphasise its left module structure and $E_A$ for $E$ 
when emphasising the right module structure. 

Thus a bi-Hilbertian bimodule is a special case of a \Cs correspondence $(E,\phi)$ over $A$, which 
is a right Hilbert $A$-module $E$ endowed with a $*$-homomorphism  
$\phi: A \rightarrow \BndE$, where $\BndE$ is the algebra of adjointable operators on 
$E$. For $x$ and
$y$ in $E_A$, we denote the associated rank-one 
operator by $\Theta_{x,y}:=x( y|\cdot)_A$. The 
algebra of compact operators $\End^{0}_A(E)$ is the 
closed linear span of the rank-one operators $\Theta_{x,y}$.
The algebra $\BndE$ is the multiplier algebra $\Mult(\End^0_A(E))$ of the 
compact endomorphisms $\End^0_A(E)$. 

Since $E$ is countably generated
(as a right module) there are vectors $\{e_j\}_{j\geq 1}\subset E$ such that
$$
\sum_{j\geq 1}\Theta_{e_j,e_j}={\rm Id}_E,
$$
where the convergence is in the strict topology of $\End^*_A(E)$. Such a
collection of vectors is called a frame, and \cite[Theorem 2.22]{KajPinWat} proves that
\begin{equation}
\ei^\beta:=\sum_{j\geq 1}{}_A(e_j|e_j)
\label{eq:fin-right}
\end{equation}
is a well-defined (central positive) element of the multiplier algebra of $A$ if and only if the left action of
$A$ on $E$ is by compact endomorphisms. The injectivity of the left action
which we assume ensures that $\ei^\beta$ is invertible (justifying the notation).
Equation \eqref{eq:fin-right} expresses
the finiteness of the right Watatani index of $E$, which is then independent of the choice of frame.
This finiteness condition seems to be the correct replacement
for the finitely generated hypothesis in the unital case, since a module over a unital algebra
with finite right Watatani index is finitely generated (and so projective). As further evidence for this, and for 
later use, we record the following result.

\begin{prop}
\label{prop:suspend}
Let $E$ be a bi-Hilbertian $A$-bimodule with finite right Watatani index $\ei^\beta$. Define the suspended 
bi-Hilbertian $\S A$-bimodule $\S E$ over the suspension $\S A:=C_0(\R)\ox A$ as follows.
Define $\S E:=C_0(\R)\ox E$, with the operations ($f_j,\,g_j\in C_0(\R)$,
$a_j\in A$, $e_j\in E$)
\begin{align*}
(g_1\ox a_1)\cdot(f\ox e)\cdot(g_2\ox a_2)&=g_1fg_2\ox a_1ea_2\\ 
(f_1\ox e_1|f_2\ox e_2)_{\S A}&:=f_1^*f_2\ox (e_1|e_2)_A\\
{}_{\S A}(f_1\ox e_1|f_2\ox e_2)&=f_1f_2^*\ox{}_A(e_1|e_2).
\end{align*}
Then $\S E$ has finite right Watatani index given by $1\ox \ei^\beta$ where $1\in C_b(\R)$ is
the constant function with value 1 and $\ei^\beta \in \Mult(A)$ is the right Watatani index of $E$. If $E_A$
is full so too is $\S E_{\S A}$ and if the left action of $A$ on $E$ is injective, so
too is the left action of $\S A$ on $\S E$.
\end{prop}
\begin{proof}
The proof that $\S E$ is bi-Hilbertian 
is a routine check of the conditions, and so too the statements about fullness and
injectivity. The right Watatani index must be finite by \cite[Theorem 2.22]{KajPinWat},
since $\S A$ acts by compacts on $\S E$, 
and so it only remains to determine the value of the index.

We let $\{e_j\}_{j\geq 1}$ be a (countable) frame for $E$ and 
pick a partition of unity $(\phi_k)_{k\in\Z}$ subordinate to the intervals $(k-\epsilon,1+k+\epsilon)$
for some fixed $0<\epsilon<1$. Then by a direct computation 
we find that $(\sqrt{\phi_k}\ox e_j)_{j,k}$ is a frame for 
$\S E_{\S A}$ and similarly that
\begin{equation*}
\sum_{j,k}{}_{\S A}(\sqrt{\phi_k}\ox e_j|\sqrt{\phi_k}\ox e_j)
=1\ox \ei^\beta \in \Mult (\S A) \simeq C_b(\R)\otimes \Mult (A).\qedhere
\end{equation*}
\end{proof}

An important class of examples are the
\smeb s (SMEBs) over $A$. A \smeb\  is a bi-Hilbertian $A$-bimodule for which
\[
{}_A(e|f)g = e (f|g)_A.
\] 
We do not require this compatibility condition 
in the definition of bi-Hilbertian bimodule.\footnote{Our notion of Hilbert bimodule is 
different from the one of \cite[Definition~1.8]{BMS}, which was 
used in \cite{AEE97} in the construction of generalised 
crossed products.} We will see in Proposition \ref{prop:scalars} that, 
upon changing the algebra of scalars, we can always 
construct a \smeb \ out of a bi-Hilbertian bimodule. 
This implies in particular that the Cuntz-Pimsner algebra of a 
bi-Hilbertian bimodule can always be interpreted as a 
generalised crossed product in the sense of \cite{AEE97} for a \smeb \ over a different algebra.

%
%

\section{Cuntz-Pimsner algebras}
\label{sec:CP-algs}

We start from a bi-Hilbertian $A$-bimodule $E$ with finite right Watatani index.
We assume that the left action of $A$ (which is necessarily by compacts) is also injective,
and that the right module $E_A$ is full.
Regarding $E$ as a right module with a left $A$-action by adjointable operators (a correspondence) 
we can construct
the Cuntz-Pimsner algebra $\Pim$. This we do concretely in the Fock representation. The algebraic
Fock module is the algebraic direct sum $$
F_E^{\textup{alg}}=\bigoplus_{k\geq 0}^{\textup{alg}}E^{\otimes_A k}
=\bigoplus_{k=0}^{\textup{alg}}E^{\otimes k}
=A\oplus E\oplus E^{\otimes 2}\oplus\cdots
$$
where the copy of $A$ is the trivial $A$-correspondence. The Fock module $F_E$ is the Hilbert \Cs module completion of $F^{\textup{alg}}_E$.
For $\nu\in F_E^{\textup{alg}}$, we define
the creation operator $T_\nu$ by the formula
$$
T_\nu(e_1\ox\cdots\ox e_k)=\nu\ox e_1\ox\cdots\ox e_k,\qquad e_j\in E.
$$
The expression $T_\nu$ extends to an adjointable operator on $F_E$, whose adjoint $T^*_\nu$
acts (when $\nu$ is homogenous with $\nu\in E^{\ox |\nu|}$) by
$$
T_\nu^*(e_1\ox\cdots\ox e_k)
=\left\{\begin{array}{ll} (\nu|e_1\ox\cdots \ox e_{|\nu|})_A\cdot e_{|\nu|+1}\ox\cdots \ox e_k
& k\geq |\nu|\\ 0 & {\rm otherwise}\end{array}\right.,
$$
and so is called an annihilation operator. 
The $C^*$-algebra generated by the set of creation operators $\{T_e: \;e\in E\}$ 
is the Toeplitz-Pimsner algebra $\T_E$. 
It is straightforward to show that $\T_E$ contains 
the algebra $\K(F_E)$ of compact endomorphisms on the Fock module as an ideal. The defining 
extension for the Cuntz-Pimsner algebra $\Pim$ is the short exact sequence
\begin{equation}
\label{eq:ext}
\xymatrix{0 \ar[r] & \K(F_{E}) \ar[r] & \Toe \ar[r]^{\pi} & \Pim \ar[r] & 0.}
\end{equation}
It should be noted that Pimsner \cite{Pimsner} in his general construction uses an ideal that 
in general is smaller than $\End^0_A(F_E)$. 
In our case, $A$ acts from the left on $E_{A}$ by compact endomorphisms, 
ensuring that Pimsner's ideal coincides with $\End^0_A(F_E)$. 
For $\nu\in \algFock$, we let $S_\nu$ denote the class of 
$T_\nu$ in $\Pim$. If $\nu\in E^{\otimes k}$ we write $|\nu|:=k$.

Since we assume $A$ to be separable and nuclear, by \cite[Theorem~2.7]{LledoVas} 
(see also \cite[Theorem~7.3]{Katsura}) the algebra $\Pim$ is separable 
and nuclear. By \cite[Corollary~IV.3.2.5]{Blackadar2} \Cs algebra extensions with separable 
and nuclear quotients are semi-split, hence the defining extension 
\eqref{eq:ext} is semi-split, 
i.e. it admits a completely positive cross section $s:\Pim \to \Toe$.
As a consequence, the above extension will induce six terms exact sequences in $KK$-theory. 

Using the natural Morita equivalence between $\K(F_E)$ and $A$, 
the $KK$-equivalence between $A$ and $\Toe$ proved in 
\cite[Theorem~4.4]{Pimsner} and \cite[Lemma~4.7]{Pimsner}, 
the six term exact sequences can be simplified to a great extent. 
Specialising to the case of $K$-theory we obtain
\begin{equation*}
\xymatrix{K_0(A) \ar[r]^-{1 - [E]} &K_0(A) \ar[r]^{\iota_*} &K_0(O_E) \ar[d]^{\partial}\\
K_1(O_E) \ar[u]^-{\partial}&K_1(A) \ar[l]^-{\iota_*}& K_1(A)\ar[l]^-{1 - [E]} 
}
\end{equation*}
where $\iota_*:=\iota_{A,O_E*}$ is the map in 
$K$-theory induced by the inclusion $\iota_{A,O_E}: A \hookrightarrow \Pim$ 
of the coefficient algebra into the Pimsner algebra and $1-[E]$ denotes the 
Kasparov product $\cdot \ox_A([{\rm Id}_{KK(A,A)}]-[E])$.

Similarly, the corresponding six term exact sequence for $K$-homology reads
\[
\xymatrix{K^0(A) \ar[d]_-{\partial} &K^0(A) \ar[l]_-{1 - [E]}  &K^0(O_E) \ar[l]^{\iota^*}\\
K^1(O_E) \ar[r]_-{\iota^*} &K^1(A)\ar[r]^-{1 - [E]} & K^1(A)  \ar[u]_{\partial}
}.
  \]
\subsection{Pimsner's extension of scalars}

Before tackling the extension, its $KK$-class and the relation to mapping cones, 
we examine the relationship of the Cuntz-Pimsner construction to the generalised crossed
product set up of \cite{AEE97}. Pimsner \cite{Pimsner} showed that by changing the scalars the completely
positive cross section mentioned above can be obtained explicitly, though this is 
at the expense of changing the exact sequence
\eqref{eq:ext} and the coefficient algebra. 

We will recall these constructions, and a little background, with a view to proving that
Pimsner's extension of scalars realises $O_E$ as the Cuntz-Pimsner algebra of a SMEB.
While at least some of the content of this statement is folklore, we could
find nothing more explicit than Pimsner's original construction in the literature. We 
provide both a precise statement and proof below.


The formula 
$$ 
z\cdot S_\nu:=z^{|\nu|}S_\nu, \qquad \forall\, \nu \in E^{\otimes k},
$$
is easily seen to extend to a $U(1)$-action on $\Pim$. 
We denote the fixed point algebra for this action by $\core$. 
Averaging over the circle action defines a conditional expectation 
$$
\rho:\Pim \to \core, \qquad \rho(x):=\int_{U(1)} z\cdot x\, \mathrm{d}z, 
$$
where $\mathrm{d}z$ 
denotes the normalized Haar measure on $U(1)$. 
The infinitesimal generator of the circle action 
defines a closed operator $N$ on the completion 
$X_{\core}$ of $\Pim$ as a $\core$-Hilbert module
in the inner product defined from $\rho$. Under the spectral subspace assumption 
(see \cite[Definition 2.2]{careyetal11}), $N$ is a self-adjoint, regular operator with 
locally compact resolvent 
whose commutators with $\{S_\nu: \;\nu\in \algFock\}$ are bounded. In particular, 
\begin{equation}
(\O_E,X_{\core},N)
\label{eq:circle-core}
\end{equation}
defines an unbounded $(\Pim,\core)$-Kasparov module, 
where $\O_E$ is the polynomial algebra
in the creation and annihilation operators $S_e,\,S_e^*$, $e\in E_A$. 

With these reminders in place, we turn to the extension of scalars. 
First, the SMEB case is precisely when we do not need to extend the scalars, for
those \Cs correspondences $(E,\phi)$ over $A$ for which
$\core=A$ can be characterised as follows.

\begin{prop}[cf. {\cite[Proposition~5.18]{Katsura}}]
\label{th:coresmeb}
Let $(E, \phi)$ be a \Cs correspondence over $A$ with left action given by compact operators,
and let $\Pim$ be the corresponding Pimsner algebra.
Then $E$ is a \smeb \ if and only if the fixed point algebra $\Pim^{\gamma}$ 
coincides with the coefficient algebra $A$.
\end{prop}
In general, $\core$ is substantially larger than $A$ and the generator of the circle action is 
insufficient for constructing an unbounded $(\O_E,A)$-Kasparov module
representing our original extension \eqref{eq:ext}.

The unbounded Kasparov module in \eqref{eq:circle-core}
gives a class in $KK^1(\Pim,\core)$, and when $E$ is a \smeb, 
this class represents the extension \eqref{eq:ext}, \cite{RRS}.
In the more general case when $\core \neq A$, Pimsner 
considered the right $\core$-module $E' := E \otimes_{A} \core$,
\cite[pp 195-196]{Pimsner}. Under some additional assumptions this enlargement of the
scalars puts us back into the \smeb\  case, where Cuntz-Pimsner algebras are known to correspond
to the generalised crossed-products of \cite{AEE97} by \cite[Theorem 3.7]{Katsura1}.

\begin{prop}
\label{prop:scalars}
Given a correpondence $(E,\phi)$, 
suppose that the module $E_A$ is full and the left action $\phi$ is 
essential\footnote{Recall that the action of $A$ on $E$ is
said to be \emph{essential} if the linear span of $\phi(A)E$ is dense in $E$.}. 
Then the module $E' := E \otimes_{A} \core$ is a bi-Hilbertian bimodule over 
$\core$ which is left and right full and which satisfies the compatibility condition
\[
{}_A(\xi | \eta)\zeta = \xi (\eta |\zeta)_{A},
\]
hence is a \smeb \ over $\core$.
The Cuntz-Pimsner algebra $O_{E}\cong O_{E'}$ agrees with the 
generalised crossed product $\core \rtimes_{E'} \Z$.
\end{prop}
We again thank Jens Kaad and Bram Mesland for fruitful discussions that 
lead to the formulation and proof of this result.
\begin{proof}
By its very definition, $E'$ is a right Hilbert $\core$ module, with right action and inner product given by the interior tensor product construction. In particular, the right inner $\core$-valued product is given by
\[
(e_1 \otimes f_1 \vert e_2 \otimes f_2)_{\core} 
:= f_1^* (e_1 \vert e_2)_A f_2 = (f_1 \vert (e_1|e_2)_Af_2)_{\core},
\quad e_1,\,e_2\in E,\ \ f_1,\,f_2\in O_E^\gamma.
\]
If the left action of $A$ on $E$ is
essential and the right inner product is full, then $E'$ is a right-full Hilbert 
$\core$-module by the following argument. 
Using the right fullness of $E_A$, 
\cite[Lemma 5.53]{RaeburnW} shows that there 
exists a sequence $y_{j}\in E$ such that for all $b\in A$ 
\[
\lim_{k\to\infty}\sum_{j=0}^{k} ( y_{j}|y_{j})_A b=b,
\]
and thus because the left $A$ action is essential 
$\lim_{k\to\infty}\sum_{j=0}^{k} ( y_{j}|y_{j})_A y=y$ for all $y\in E$.
Now let $S_{\mu_{1}\cdots \mu_{n}}S_{\nu_{1}\cdots\nu_{n}}^{*}\in \core$. We want to show 
that this element of the fixed point algebra $\core$ can be approximated by inner products. By rewriting
the inner product
\begin{align*}
( S_{y_{j}}S_{\nu_{1}\cdots \nu_{n}}S_{\mu_{1}\cdots\mu_{n}}^{*}\,|\,S_{y_{j}})_{\core}
&=S_{\mu_{1}\cdots \mu_{n}}S_{y_{j}\nu_{1}\cdots\nu_{n}}^{*}S_{y_{j}}
=S_{\mu_{1}\cdots \mu_{n}}(S_{y_{j}}^{*}S_{y_{j}}S_{\nu_{1}\cdots \nu_{n}})^{*}\\
&=S_{\mu_{1}\cdots \mu_{n}}(\,( y_{j}|y_{j})_A S_{\nu_{1}\cdots \nu_{n}})^{*}
=S_{\mu_{1}\cdots \mu_{n}}(S_{( y_{j}|y_{j})_A \nu_{1}\cdots \nu_{n}})^{*},
\end{align*}
we see that
\[
\lim_{k\to\infty}\sum_{j=0}^{k}( S_{y_{j}}S_{\nu_{1}\cdots \nu_{n}}S_{\mu_{1}\cdots\mu_{n}}^{*}|S_{y_{j}})_{\core}
=\lim_{k\to\infty}\sum_{j=0}^{k}S_{\mu_{1}\cdots \mu_{n}}(S_{( y_{j}|y_{j})_A \nu_{1}\cdots \nu_{n}})^{*}
=S_{\mu_{1}\cdots \mu_{n}}S_{\nu_{1}\cdots \nu_{n}}^{*},
\]
and so $E\otimes_{A}\core$ is right full.

The non-trivial part is the left module structure. 
We define a left action $\tilde{\phi}: \core \to \End^*_{\core}(E')$ 
by using the natural inclusion $E' \hookrightarrow \Pim$ given 
on simple tensors by $e\otimes f \mapsto S_e \cdot f$, $e\in E$ and $f\in O_E^\gamma$.
The core $\core$ is generated by elements of the form $S_{\bf{\mu}}S_{\bf{\nu}}^*$, with 
$\vert \mu \vert = \vert \nu \vert =n$. Such elements act on simple tensors by
\[
\tilde{\phi}(S_{\mu}S_{\nu}) (e\otimes f) 
=\mu_1 \otimes (S_{\mu_2 \cdots \mu_n} S^*_{\nu_2\cdots\nu_{n}} (\nu_{1}\vert e)_A f),
\quad e\in E,\ f\in O_E^\gamma,
\]
since $S_{\mu_2 \cdots \mu_n} S_{\nu_2\cdots\nu_{n}}^{*}$ 
is again an element of the fixed point algebra $\Pim^\gamma$.

In order to define a left inner product, we again use the above identification and define
\[
_{\core}(e_1 \otimes f_1 \vert e_2\otimes f_2) := S_{e_1}f_1 f_2^* S_{e_2}^*.
\]
We now show this  inner product is left-full. This can be done by 
choosing a frame $(x_{i})_{i=1}^N$ for $E_A$ ($N$ can be infinity).
Then
\[
\sum_{i=1}^{N} S_{x_{i}}S^{*}_{x_{i}}S_{\mu_{1}\cdots \mu_{n}}S_{\nu_{1}\cdots \nu_{n}}^{*}
= S_{\mu_{1}\cdots \mu_{n}}S_{\nu_{1}\cdots \nu_{n}}^{*},
\]
and at the same time writing $\nu=\nu_1\overline{\nu}$ and $\mu=\mu_1\overline{\mu}$ we have
\[
S_{x_{i}}S^{*}_{x_{i}}S_{\mu_{1}\cdots \mu_{n}}S_{\nu_{1}\cdots \nu_{n}}^{*}
= S_{x_i}{}_{\core}( \, (x_i|\mu_1)_A\,| S_{\overline{\nu}}S_{\overline{\mu}}^*)S_{\nu_1}^*
={}_{\core}(x_i\ox(x_i|\mu_1)_A\,|\,\nu_1\ox S_{\overline{\nu}}S_{\overline{\mu}}^*),
\]
which shows that the left inner product is full.
We conclude by checking the compatibility condition by computing that
\begin{align*}
\tilde{\phi}\left(_{\core}(e_1 \otimes f_1 \vert e_2\otimes f_2)\right) e_3 \otimes f_3
& = \tilde{\phi}(S_{e_1}f_1 f_2^* S_{e_2}^*)e_3 \otimes f_3 
=e_1 \otimes (f_1 f_2^* S_{e_2}^* S_{e_3} f_3)  \\
& = e_1 \otimes f_1 \,(f_2^*  S_{e_2}^* S_{e_3}  f_3) 
=e_1 \otimes f_1\,(e_2\otimes f_2 \vert e_3 \otimes f_3)_{\core}.\qedhere
\end{align*}
\end{proof}

\subsection{The extension class}

A Kasparov module representing the class of the extension \eqref{eq:ext}
was constructed in \cite{RRS}, under the assumption that $A$ is unital and $E$ finitely generated,
and a further assumption discussed below. 
Here we recall the salient points, and extend the discussion to handle the non-unital situation. 

So we suppose that $E$ is a bi-Hilbertian $A$-bimodule with finite right Watatani index, full
as a right module and with injective left action of $A$.
We choose a frame $(e_i)_{i\geq 1}$ for $E_A$. 
The frame $(e_i)_{i\geq 1}$ induces a frame for $E^{\ox_A k}$, namely
$(e_\rho)_{|\rho|=k}$ where $\rho$ is a multi-index and $e_\rho=e_{\rho_1}\ox\cdots\ox e_{\rho_k}$.
We define 
$$
\Phi_k:\End_A^{00}(E^{\ox_A k})\to A,
\qquad \Phi_k(T)=\sum_{|\rho|=k}{}_A(Te_\rho|e_\rho).
$$
Here $\End_A^{00}(E^{\ox_A k})$ denotes the finite rank 
operators on $E^{\ox_A k}$. It follows from \cite[Lemma 2.16]{KajPinWat}
that $\Phi_k$ does not depend on the choice of frame and 
extends to a norm continuous map on $\End^0_A(E^{\ox k})$, \cite[Corollary 2.24]{KajPinWat}. 
By  \cite[Proposition 2.27]{KajPinWat}, the functionals $\Phi_k$ extend to strictly continuous maps
$\Phi_k:\End^*_A(E^{\ox k})\to \Mult(A)$.

In particular, we denote by 
$\mathrm{e}^{\beta_k}$ the element 
$\Phi_k({\rm Id}_{E^{\ox_A k}}) = \sum_{|\rho|=k}{}_A(e_\rho|e_\rho)\in {\rm Mult}(A)$. 
Since $\Phi_k$ is independent of the
choice of frame, so is $\mathrm{e}^{\beta_k}$. Note that $\mathrm{e}^{\beta_k}$ is a 
positive, central, invertible element of $\Mult(A)$, \cite[Corollaries 2.24, 2.28]{KajPinWat}. 
Therefore $\beta_k$ is a well-defined 
self-adjoint central element in $\Mult(A)$.

We further extend the functional $\Phi_k$ to $\Phi_{k}: \End^*_A(F_E)\to {\rm Mult}(A)$ 
by defining $\Phi_k(T):=\Phi_k(P_kTP_k)$ for $T\in \End^*_A(F_E)$,
where $P_k:F_E\to E^{\ox k}$ is the projection. 
Naively, we would like to define
\begin{equation}
\label{phiinftydef}
\Phi_\infty(T)\,``\!:=\!"\,{\rm res}_{s=1}\sum_{k=0}^\infty\Phi_k(T)\mathrm{e}^{-\beta_k}(1+k^2)^{-s/2},
\quad\mbox{for suitable $T\in \End_A^*(F_E)$.}
\end{equation}
Indeed, $\Phi_k(T)\mathrm{e}^{-\beta_k}$ is easily 
shown to be bounded, and  so it is tempting to try to
define $\Phi_\infty$ using some `generalised residue'
in the sense of generalised limits and Dixmier traces. In general, 
problems arise since $\Phi_\infty$ (if well-defined) is
not a numerical functional, but $A$-valued. Worse still, in the non-unital setting we
only have the strict continuity of the $\Phi_k$ in general. The lack of norm continuity
is handled as follows.

\begin{lemma}
\label{lem:toeplitz-compact}
Suppose that $T\in \T_E\subset \End_A^*(F_E)$. Then for each $k=0,\,1,\,2,\dots$,
the compression $P_kTP_k$ is a compact endomorphism on $E^{\ox k}$, and hence
$\Phi_k:\T_E\to A$ is norm continuous.
\end{lemma}
\begin{proof}
We approximate $T\in \T_E$ in norm by a finite sum of generators $T_\xi T_\eta^*$ for $\xi,\,\eta\in F_E$
homogenous. If $|\xi|\neq |\eta|$ then $P_kT_\xi T_\eta^*P_k=0$, and so we suppose that
$|\xi|= |\eta|$.

In that case, for $k<|\xi|$ we again have $P_kT_\xi T_\eta^*P_k=0$, while for $k\geq |\xi|$
the endomorphism $P_kT_\xi T_\eta^*P_k$ coincides with a compact endomorphism of
$E^{\ox k}$ by \cite[Corollary 3.7]{Pimsner} and the injectivity of the left action of $A$. Since
$P_kTP_k$ is approximated in norm by finite sums of endomorphisms $P_kT_\xi T_\eta^*P_k$,
$P_kTP_k$ is a compact endomorphism of $E^{\ox k}$.
\end{proof}

Thus for ${\rm Re}(s)>1$, since $\Vert\Phi_k(T)\mathrm{e}^{-\beta_k}\Vert\leq \Vert T\Vert$, the map
$$
\T_E\ni T \mapsto \sum_{k=0}^\infty\Phi_k(T)\mathrm{e}^{-\beta_k}(1+k^2)^{-s/2}
$$
is norm continuous. The only remaining problem with  
the tentative definition in Equation \eqref{phiinftydef} is the existence of the residue.
Following \cite{RRS}, we work under the following 
assumption guaranteeing that the residue exists for 
$T\in\T_E$.
\begin{ass}
\label{ass:one}
We assume that for every $k\in \N$, there is a $\delta>0$ such that whenever 
$\nu\in E^{\otimes k}$ there exists a $\tilde{\nu}\in E^{\otimes k}$ satisfying
$$
\Vert \mathrm{e}^{-\beta_n}\nu \mathrm{e}^{\beta_{n-k}}-\tilde\nu\Vert 
= O(n^{-\delta}),\quad\mbox{as $n\to \infty$}.
$$
\end{ass}
When Assumption \ref{ass:one} holds,
Equation \eqref{phiinftydef} defines an $A$-bilinear functional $\Phi_\infty:\Toe\to A$,
which is a continuous $A$-bilinear positive expectation, which in addition vanishes on 
$\K(E)\subset \Toe$. Hence $\Phi_\infty$ descends to a positive $A$-bilinear expectation
$\Phi_\infty:O_E\to A$.
The details of this construction can be found in \cite[Section 3.2]{RRS}, and the only change in the non-unital case is the norm continuity, which follows from Lemma \ref{lem:toeplitz-compact}.
This functional furnishes us with an $A$-valued inner product
$(S_1|S_2)_A:=\Phi_\infty(S_1^*S_2)$ on $O_E$, and the completed module
is denoted $\phimod$. 

{\bf We assume that Assumption \ref{ass:one} holds for the remainder of the paper.}

\begin{thm}[Theorem 3.14 of \cite{RRS}]
\label{thm:RRS-ext} If the bi-Hilbertian bimodule $E$ satisfies
Assumption \ref{ass:one}, then the tuple $(O_E,\phimod,2Q-1)$ is an 
odd Kasparov module representing
the class of the extension \eqref{eq:ext}. The projection $Q$ 
has range isometrically isomorphic to the Fock module $F_E$.
\end{thm}

\begin{ex}
\label{smebexample}
When $E$ is a \smeb, $\Phi_{\infty}:O_{E}\to A$ coincides with 
the expectation $\rho:O_{E}\to \core$ discussed prior to Equation \eqref{eq:circle-core}. Therefore
$$
\phimod=\bigoplus_{n\in\Z}E^{\ox n}
$$
with the convention that $E^{\ox(-|n|)}=\overline{E}^{\ox |n|}$, 
where $\overline{E}$ is the conjugate module, which agrees 
with the \Cs algebraic dual of $E$. In this case we can define the {\bf number operator} $N$ 
on the module $\phimod$ by $N\rho=n\rho$ for $\rho\in E^{\ox n}$. 
Then $(\O_E, {\phimod},N)$ is an unbounded 
Kasparov module representing the class of the extension \eqref{eq:ext} in $KK^1(O_E,A)$, by
\cite[Theorem 3.1]{RRS}.
\end{ex}

In \cite{kappa}, under an additional assumption, 
Theorem \ref{thm:RRS-ext} was extended, presenting an unbounded representative of the class defined by
$(O_E,\phimod,2Q-1)$. In order to construct the unbounded representative $\D$ we need an additional assumption on
the bimodule. Under Assumption \ref{ass:one}, we can define the 
operator $\topop_k:E^{\otimes k}\to E^{\otimes k}$ by
$$
\topop_k\nu:=\tilde{\nu}=\lim_{n\to\infty}\mathrm{e}^{-\beta_n}\nu \mathrm{e}^{\beta_{n-k}}.
$$
By \cite[Lemma 2.2]{kappa}, each $\topop_k$ is adjointable for both module structures,
a bimodule map and positive. 
Then in order to construct $\D$ we need to assume
\begin{ass}
\label{ass:two}
For any $k$, we can write $\topop_k=c_kR_k=R_{k}c_{k}$ where 
$R_{k}\in \End^{*}_{A}(E^{\otimes k})$ is a projection and $c_k$ 
is given by left-multiplication by an element in $\Mult(A)$. 
\end{ass}

Both Assumptions \ref{ass:one} and \ref{ass:two} hold for a wide variety of examples,
as shown in \cite{RRS} and \cite{kappa}.

When $A$ is unital and Assumption 2 holds, \cite[Theorem 2.10]{kappa} proves that the 
module $\phimod$ decomposes as a direct sum of 
bi-Hilbertian $A$-bimodules $\Xi_{n,r}$ of finite right Watatani index. 

To check this in the non-unital case means computing the index
directly in terms of the frame for $\phimod$ presented in 
\cite[Lemmas 2.8, 2.9]{kappa}. The construction of the frame begins with
a frame $\{e_j\}_{j\geq 1}$ for $E_A$ and a frame $\{f_k\}_{k\geq 1}$ for ${}_AE$,
and produces a frame $\{W_{e_\rho,c_{|\sigma|}^{-1/2}f_\sigma}\}_{\rho,\sigma}\subset \phimod$ for
multi-indices $\rho,\,\sigma$, and where $c_{|\sigma|}$ is as in Assumption \ref{ass:two}. 
For fixed values of $|\rho|,\,|\sigma|$ we have
\begin{align*}
\sum_{|\rho|=r,\,|\sigma|=s}{}_A\Big(W_{e_\rho,c_{s}^{-1/2}f_\sigma}\Big|W_{e_\rho,c_{s}^{-1/2}f_\sigma}\Big)
&=\sum_{|\rho|=r,\,|\sigma|=s}\Phi_\infty(S_{e_\rho}S^*_{c_{s}^{-1/2}f_\sigma}
S_{c_{s}^{-1/2}f_\sigma}S_{e_\rho}^*)\nonumber\\
&=\sum_{|\rho|=r,\,|\sigma|=s}\Phi_\infty(S_{e_\rho}(c_{s}^{-1/2}f_\sigma|
c_{s}^{-1/2}f_\sigma)_AS_{e_\rho}^*)\nonumber\\
&\leq\Vert c_{s}^{-1}\Vert \sum_{|\rho|=r,|\sigma|=s}\Phi_\infty(S_{e_\rho}(f_\sigma|f_\sigma)_AS_{e_\rho}^*)
\nonumber\\
&\leq\Vert c_{s}^{-1}\Vert \sum_{|\rho|=r}\Phi_\infty(S_{e_\rho}\ell_sS_{e_\rho}^*)
\leq\Vert c_{s}\Vert^{-1} \ell_s\sum_{|\rho|=r}\Phi_\infty(S_{e_\rho}S_{e_\rho}^*)\nonumber\\
&=\Vert c_{s}\Vert^{-1}\, \ell_s\, \ei^{\beta_r}
\end{align*}
where $\ell_s$ is the {\em left numerical} Watatani index of $E^{\ox s}$, which is finite by 
\cite[Theorem 4.8]{KajPinWat}.
This computation shows (in particular) that the summands $\Xi_{n,r}$ in the decomposition
\begin{equation}
\phimod=\bigoplus_{n\in\Z,\ r\geq\max\{0,n\}}\Xi_{n,r}
\label{eq:en-ar}
\end{equation}
are bi-Hilbertian $A$-bimodules of finite right Watatani index,
and we denote the projections onto these sub-modules by $P_{n,r}$.

Then one defines $\D=\sum_{n,r}\psi(n,r)P_{n,r}$ 
where $\psi$ is a suitable function, \cite[Definition 2.12]{kappa}. 
By \cite[Lemma~2.14]{kappa} the projection $Q$ appearing in 
Theorem~\ref{thm:RRS-ext} has the form
\[
Q=\sum_{n=0}^{\infty} P_{n,n},
\]
with respect to the above decomposition \eqref{eq:en-ar}.

\begin{rmk}
Note that in the \smeb \ case, a suitable choice of the operator $\D$ along with the decomposition of
the module $\phimod$ coincides with the number operator and the decomposition described in 
Example \ref{smebexample}.
\end{rmk}

{\bf We assume that Assumption \ref{ass:two} holds for the remainder of the paper.}

\begin{thm}[Theorem 2.16 of \cite{kappa}]
\label{thm:kappa-ext} If the bi-Hilbertian $A$-bimodule $E$ satisfies
Assumptions \ref{ass:one} and \ref{ass:two}, 
then the tuple
$(O_E,\phimod,\D)$ is an odd unbounded Kasparov module representing
the class of the extension \eqref{eq:ext}. The spectrum of $\D$ can be chosen to consist of
integers with  bi-Hilbertian $A$-bimodule eigenspaces of finite right Watatani index, 
and non-negative spectral projection  $Q$.
\end{thm}

The only difference arising in the non-unital case is that the resolvent of $\D$ is
not compact, but only locally compact. This follows since, 
just as in Lemma \ref{lem:toeplitz-compact}, the compression $P_{m,s}SP_{n,r}$
of  $S\in O_E$ is a compact endomorphism. Since the eigenvalues of
$\D$ are chosen to have $\pm$infinity as their only limit points, we find that
$S(1+\D^2)^{-1/2}$ is a norm convergent sum of compacts.

Our last task before turning to the mapping cone exact sequence is to show that 
the class of modules we consider is stable under suspension. Proposition \ref{prop:suspend}
gives us most of what we want, and we just need to check that if $E$ satisfies Assumptions
\ref{ass:one} and \ref{ass:two} then so too does $\S E$.

\begin{prop}
\label{prop:suspend-assumps} If $E$ is a bi-Hilbertian $A$-bimodule of finite right
Watatani index which is full as a right module and with injective left action satisfying
Assumptions \ref{ass:one} and \ref{ass:two}, then
the suspended module $\S E$ is a bi-Hilbertian $\S A$-bimodule of finite right
Watatani index which is full as a right module and with injective left action
satisfying
Assumptions \ref{ass:one} and \ref{ass:two}.
\end{prop}
\begin{proof}
This follows from Proposition \ref{prop:suspend} and 
the fact that the right Watatani index of $(\S E)^{\ox k}$ is $1\ox \ei^{\beta_k}$ 
where $\ei^{\beta_k}$ is the 
right Watatani index of $E^{\ox k}$. 
\end{proof}

\section{Comparing the mapping cone and Cuntz-Pimsner exact sequences}
\label{sec:compare}

In addition to the defining exact sequence for $O_E$, 
we can look at the mapping cone extension for the inclusion 
$\iota_{A,O_E}: A \hookrightarrow \Pim$ of the scalars into the Cuntz-Pimsner algebra. Recall that the 
mapping cone $M(A,\Pim)$ of the inclusion $\iota_{A,O_E}$ is the \Cs algebra
\[
M(A,\Pim) := \big\lbrace f \in C([0,\infty),O_E) \ : f(0)\in A,\, f(\infty) =0,\,f\ \mbox{continuous}.\big\rbrace
\]
We will frequently abbreviate $M(A,\Pim)$ to $M$.
The algebra $M$ fits into a short exact sequence 
involving the suspension $S\Pim \simeq C_0((0,\infty),A)$. Due to the use of
the mapping cone, we will often write suspensions as $C_0((0,\infty))\otimes\cdot$ instead
of $C_0(\R)\otimes\cdot$. The sequence is
\[
\xymatrix{0 \ar[r] &\S \Pim \ar[r]^-{j_*}& M(A,O_E)\ar[r]^-{\ev} & A \ar[r] &0,}
\]
where $\ev(f) = f(0)$ and $j(g \otimes a) (t) = g(t)a $. The mapping cone 
extension is semi-split and induces six term exact sequences in $KK$-theory. 

Specialising to $K$-theory yields the exact sequence 
\begin{equation*}
\label{eq:coneKTh}
\xymatrix{K_0(A) \ar[r]^-{\partial'} & K_0(\S\Pim) \ar[r]^-{j_*} & K_1(M) \ar[d]^-{\ev_*}\\
K_0(M) \ar[u]^{\ev_*}&K_1(\S \Pim)\ar[l]^-{j_*}&K_1(A)\ar[l]^-{\partial'}}
\end{equation*}
By \cite[Lemma~3.1]{CPR1} the boundary map $\partial' : K_j(A)\to K_{j+1}(\Pim)$
is given, up to the Bott map $\bott: K_{j}(\Pim) \to K_{j+1}(\S\Pim)$, 
by minus the inclusion of $A$ in $O_E$, i.e. $\partial' = -\bott \circ \iota_{A,O_E*}$. 
Similar considerations hold for the dual $K$-homology exact sequence.

We now compare the defining short exact sequence for $O_E$ and the
mapping cone sequence for the inclusion $\iota_{A,O_E}:A\hookrightarrow O_E$. 
To do so, we use the identification $\bott: K_{j}(\Pim) \to K_{j+1}(\S\Pim)$ 
to define a map $j_*^B: K_i(\Pim) \to K_{i+1}(M)$ 
given by $j_* \circ \bott$. Then we have the 
partial comparison with two out of three maps given by  the identity:
\[
\xymatrix{\cdots \ar[r]^{\iota_{*}}&K_0(\Pim) \ar[r]^-{j_*^B}\ar[d]^{=} &K_{1}(M) \ar[r]^-{\ev_*} \ar[d]^{?}&
K_{1}(A) \ar[r]^{\iota_{*}}\ar[d]^{=}& K_1(\Pim) \ar[r]^-{j_*^B}\ar[d]^{=} &K_{0}(M) \ar[r]^-{\ev_*} \ar[d]^{?}&
K_{0}(A) \ar[r]^{\iota_{*}}\ar[d]^{=}&\cdots \\
\cdots \ar[r]^{\iota_{*}}&K_0(\Pim) \ar[r]^-{\partial} &K_{1}(A) \ar[r]^-{1-[E]} &
K_{1}(A) \ar[r]^{\iota_{*}}&K_1(\Pim) \ar[r]^-{\partial} &K_{0}(A) \ar[r]^-{1-[E]} &
K_{0}(A) \ar[r]^{\iota_{*}}&\cdots }
\]
Thus the question we seek to address is whether there is a map
that can be put in place of $?$ which  makes the diagram commute 
(and so provide an isomorphism of six-term sequences).

\begin{rmk}
As pointed out in the introduction, the existence of an 
isomorphism between the two exact sequences 
follows from the fact that the $KK$-category is triangulated, 
with exact triangles the mapping cone triangles. The missing map can
be easily constructed as a Kasparov product with the class 
\begin{equation}
[\tilde{\alpha}]\otimes_{M_{\pi}}[u]\otimes_{\K(\Fock)}[\Fock] \in KK(M,A),
\label{eq:class}
\end{equation}
where $M_\pi$ denotes the mapping cone $M(\Toe, \Pim)$, $\tilde{\alpha}: M(A,\Pim) \to M(\Toe,\Pim)$ is the inclusion of 
mapping cones induced by the natural inclusion $\alpha:A \to \Toe$, $[\Fock] \in KK(\K(\Fock),A)$ 
is the class of the Morita equivalence, and $[u] \in KK(M(\Toe,\Pim),\K(\Fock))$ 
is the $KK$-equivalence given by \cite[Corollary~2.4]{CuSk86}.


In the following we will provide an unbounded 
representative for a class that makes diagrams in $K$-theory commute, by lifting 
the unbounded representative of the extension class to the mapping cone, 
as we describe below. The axioms of triangulated categories do not 
guarantee the uniqueness of such a class, hence we leave it as an 
open problem to verify that our unbouded Kasparov module is a 
representative for the class in \eqref{eq:class}
\end{rmk}

The map $\partial$ is implemented by the Kasparov product with
the class of the defining extension.
Now we are working under Assumptions \ref{ass:one} and \ref{ass:two}, 
and so we have
an explicit unbounded representative $(\O_E,\phimod,\D)$ for the defining extension. 
As noted earlier, $\D$ has discrete spectrum and commutes 
with the left action of $A$, hence we have $\iota_{A,O_E}^*[(\O_E,\phimod,\D)]=0$. 
In particular there is a class $[\hat\D]\in KK(M(A,O_E),A)$ such that 
$j^{B*}[\hat\D]=[(\O_E,\phimod,\D)]$. As the notation suggests, there is an explicit
unbounded representative for the class $[\hat\D]$, provided by
the main result of \cite{CPR1}. 
Subject to some further hypotheses, the class $[\hat\D]$ 
can be used to help compute index pairings, \cite[Theorem 5.1]{CPR1}, because of the 
explicit unbounded representative.
The even unbounded Kasparov module representing the class $[\hat\D]$ is denoted
\begin{equation}
\big(M(A,O_E),\,\widehat{\Xi}_A=X\oplus X^\sim,\, \hat\D\big).
\label{eq:hat-dee}
\end{equation}
The module $X$ is a completion of $L^2([0,\infty))\ox\phimod$
while $X^\sim$ also contains functions with a limit at infinity. The operator is given by
$$
\hat\D=\begin{pmatrix} 0 & -\partial_t+\D\\ \partial_t+\D & 0  \end{pmatrix},
$$
together with suitable APS-type boundary conditions, \cite[Section 4.1]{CPR1}. 
The details will not influence the
following discussion, but we stress that the operator is concrete, and so index pairings are explicitly computable.

Trying $\cdot\hox_{M}\hat{\D}$ in place of $?$ we find
that the squares to the left of each instance of $\hat{\D}$ in the diagram
\begin{align}\xymatrix{\cdots \ar[r]^{\iota_{*}}&K_0(\Pim) \ar[r]^-{j_*^B}\ar[d]^{=} &K_{1}(M) \ar[r]^-{\ev_*} \ar[d]^{\cdot\ox[\hat\D]}&
K_{1}(A) \ar[r]^{\iota_{*}}\ar[d]^{=}& K_1(\Pim) \ar[r]^-{j_*^B}\ar[d]^{=} &K_{0}(M) \ar[r]^-{\ev_*} \ar[d]^{\cdot\ox[\hat\D]}&
K_{0}(A) \ar[r]^{\iota_{*}}\ar[d]^{=}&\cdots \\
\cdots \ar[r]^{\iota_{*}}&K_0(\Pim) \ar[r]^-{\partial} &K_{1}(A) \ar[r]^-{1-[E]} &
K_{1}(A) \ar[r]^{\iota_{*}}&K_1(\Pim) \ar[r]^-{\partial} &K_{0}(A) \ar[r]^-{1-[E]} &
K_{0}(A) \ar[r]^{\iota_{*}}&\cdots }\label{eq:diagram}
\end{align}
commute. Now what about the squares to the right? 

\section{The $K$-theory of the mapping cone of a Cuntz-Pimsner algebra}

We use the characterisation of the $K$-theory group $K_0(M)$ due to \cite{Putnam}. Classes
in $K_0(M)$ can be realised as (stable homotopy classes of) partial isometries $v\in M_k(O_E)$
with range and source projections $vv^*,\,v^*v\in M_k(A)$. In the usual projection picture, the
class of the partial isometry $v$ corresponds to the class \cite[Section 5]{CPR1}
$$
[e_v]-\left[\begin{pmatrix} 1 & 0\\ 0 & 0\end{pmatrix}\right],\qquad 
e_v(t)=\begin{pmatrix} 1-\frac{1}{1+t^2}vv^* & \frac{-it}{1+t^2}v\\ \frac{it}{1+t^2}v^* & \frac{1}{1+t^2}v^*v\end{pmatrix}.
$$
It is important to note for later use that this characterisation of 
$K_0(M)$ does not depend on having unital algebras: 
for $A\subset B$ not necessarily unital, we need
only consider partial isometries $v\in M_k(\tilde{B})$ over the unitisation $\tilde{B}$
with $vv^*,\,v^*v\in M_k(\tilde{A})$ satisfying $vv^*-v^*v\in M_k(A)$
or even more generally, $[vv^*]-[v^*v]\in K_0(A)$. In the following discussion
one can just replace $v$ over $O_E$ with $v$ over $\tilde{O_E}$ 
satisfying $[v^*v]-[vv^*]\in K_0(A)$, and even take $v\in\widetilde{\mathcal{K}\ox O_E}$.

Returning to the exact sequence, we again let $v$ be a 
partial isometry over $\Pim$,  say $v\in M_k(O_E)$, with $v^*v$ and $vv^*$ projections 
over $\iota_{A,O_E}(A)$. Then we have
$\ev_*([v])=[v^*v]-[vv^*]$. In the other direction, we need to evaluate the 
product $[v]\ox_{\Pim}[\hat{\D}]\ox_A([{\rm Id}_{KK(A,A)}]-[E])$.

Our strategy is to use \cite[Theorem~5.1]{CPR1}, to find that the latter product is given by
\begin{equation}
-{\rm Index}(Q_kvQ_k:v^*vF_E^k\to vv^*F_E^k)\ox_A([{\rm Id}_{KK(A,A)}]-[E]),
\label{eq:index-cpr1}
\end{equation}
where $Q_k=Q\otimes 1_k$, and $Q\phimod=F_E$, the Fock module. Here $[E]$ is short hand for
the class in $KK(A,A)$ of $(A,E_A,0)$, and similarly $[{\rm Id}_{KK(A,A)}]$ can be represented by
$(A,A_A,0)$.

In order to be able to use this formula, we need to check the hypotheses 
of \cite[Theorem~5.1]{CPR1}, and then actually compute
the product in Equation \eqref{eq:index-cpr1}. The precise statement
of \cite[Theorem~5.1]{CPR1} in our case is

\begin{thm}
Let $(\O_E,\phimod,\D)$ be the unbounded Kasparov module
for the (pre-) $C^*$-algebras $\O_E\subset O_E$ and $A$ 
representing the extension class.
Let $(M,\widehat{\Xi}_A,\hat\D)$ be the unbounded
Kasparov $M(A,O_E)$-$A$ module of Equation \eqref{eq:hat-dee}.
Then for any unitary $u\in M_k(A)$ such that 
$Q_k$ and the projection $(\ker\D)\ox{\rm Id}_k$ both commute with
$u(\D\ox {\rm Id}_k)u^*=:u\D_k u^*$ and $u^*\D_k u$
we have the following equality of
index pairings with values in $K_0(A)$:
\begin{align*} 
\langle [u],[(\O_E,\phimod,\D)]\rangle&:={\rm Index}(Q_ku^*Q_k)
=
{\rm Index}(e_u(\hat\D_k\otimes1_2) e_u)-{\rm Index}(\hat\D_k)\\
&=:
\left\langle [e_u]-\left[\begin{pmatrix} 1 & 0\\ 0 & 0\end{pmatrix}\right],[(M,\hat{\Xi}_A,\hat\D)]\right\rangle
\in K_0(A).
\end{align*}
Moreover, if $v$ is a  partial isometry, $v\in M_k(\O_E)$, with
$vv^*,v^*v\in M_k(A)$ and such that 
$Q_k$ and $ (\ker\D)_k$ {both  commute  with} 
$v\D_k v^*$ and $ v^*\D_k v$
we have
\begin{align*}
\left\langle [e_v]-\left[\begin{pmatrix} 1 & 0\\ 0 & 0\end{pmatrix}\right],[(M,\hat{\Xi}_A,\hat{\D})]\right\rangle
&=
-{\rm Index}\big(Q_kvQ_k:v^*vF_E^k\to vv^*F_E^k\big)\in K_0(A).
\end{align*}
\end{thm}
It is important to observe that when we consider $v\D_kv^*$, we are suppressing the representation
of $v$ on $\phimod^k$, but it makes a difference in what follows. For this reason
we temporarily introduce the notation $\varphi:\Pim \to \End^*_A(\phimod)$ 
for the representation. This representation naturally extends to a representation
$\varphi_k : M_k(\Pim) \to \End^*_A(\oplus_{i=1}^k \phimod)$, 
and to $\mathcal{K}\ox \Pim\to \End^*_A(\H\ox\phimod)$ in the countably generated case. 

The next result, that we state here in the particular case of Cuntz-Pimsner 
algebras of bimodules satisfying Assumptions \ref{ass:one} and \ref{ass:two}, 
holds in general for any representation of an algebra on a bimodule, 
for which there exists a decomposition of the type in Equation \eqref{eq:en-ar}.
\begin{lemma}
\label{lem:chop-vee}
Given $v\in M_k(O_E)$ or $\mathcal{K}\ox O_E$, define
$$
v_{m,s}:=\sum_{n\in\Z,r\geq\max\{0,n\}}P_{n+m,r+s}\varphi_k(v)P_{n,r}.
$$
Then $\varphi_k(v)=\sum_{m,s}v_{m,s}$ where the sum converges strictly.
If $v$ is a partial isometry with range and source projections in $A$,
the $v_{n,r}$ are partial isometries with 
$v_{n,r}^*v_{m,s}=\delta_{n,m}\delta_{r,s}v_{n,r}^*v_{n,r}$ and
$v_{n,r}v_{m,s}^*=\delta_{n,m}\delta_{r,s}v_{n,r}v_{n,r}^*$. Hence the projections $v_{n,r}^*v_{n,r}$
are pairwise orthogonal, and likewise the projections $v_{n,r}v_{n,r}^*$ are pairwise orthogonal.
\end{lemma}
\begin{proof}
The first statement follows from the definition of $v_{m,s}$, the orthogonal decomposition 
$\phimod=\oplus\Xi_{n,r}$ together with the fact that $\sum_{n,r} P_{n.r}$ converges to $\Id_{\phimod}$ strictly.

Now suppose that we have $v\in O_E$ a partial isometry with range and source projections
in $A$ (the following argument adapts to partial isometries in matrix algebras).
Then we have
$$
\varphi(v)=\sum v_{m,s}.
$$
Since $vv^*$ and $v^*v$ are matrices over $A$, we see, in particular, 
that they commute with $\D$. Hence
$$
vv^*=\sum v_{m,s}v_{n,r}^*\in M_N(A) \Rightarrow v_{m,s}v_{n,r}^*=\delta_{m,n}\delta_{s,r}v_{m,s}v_{m,s}^*.
$$
Similarly $v_{m,s}^*v_{n,r}=\delta_{m,n}\delta_{s,r}v_{m,s}^*v_{n,r}$. 
Now we recall that $vv^*$ is a projection and consider
$$
vv^*=\sum v_{m,s}v_{m,s}^*=(vv^*)^2=(\sum v_{m,s}v_{m,s}^*)^2
=\sum v_{m,s}v_{m,s}^*v_{n,r}v_{n,r}^*=\sum (v_{m,s}v_{m,s}^*)^2
$$
the last equality following from $v_{m,s}^*v_{n,r}=\delta_{m,n}\delta_{s,r}v_{m,s}v_{m,s}^*$. 
Since $v_{m,s}v_{m,s}^*v_{n,r}v_{n,r}^*=0$ for
$(m,s)\neq (n,r)$, we see that each $v_{m,s}v_{m,s}^*$ 
is a projection over $A$, and the various $v_{m,s}v_{m,s}^*$ 
are mutually orthogonal. Similarly, the $v_{m,s}^*v_{m,s}$ form a set of mutually orthogonal
projections. 
\end{proof}
We deduce the commutation relation 
$v_{m,s}P_{l,t}=P_{l+m,t+s}v_{m,s}$ for all $l,m\in \Z$, $t\geq \max\{0,l\}$, $s\geq \max\{0,m\}$. This 
seems surprising given the more complicated commutation relation of \cite[Lemma 2.15]{kappa}, but
they are reconciled by the following observation (proved in the Lemma below). 
If $\mu\in F_E$ is homogenous of degree $|\mu|$
then for $n$ sufficiently large and positive
$$
P_{n+|\mu|,n+|\mu|}S_\mu P_{n,n}\neq 0.
$$
Hence $S_\mu=\sum_{j=0}^{|\mu|}(S_\mu)_{|\mu|,j}$ has $S_{\mu,|\mu|}\neq 0$, 
and we see that the decomposition in the Lemma 
uses much more information than just the degree given by the gauge action.

\begin{lemma}
\label{lem:homog}
Suppose that $S\in O_E$ satisfies $S_{n,r}\neq 0$ for some 
$n\in \Z$ and $r\geq \max\{0,n\}$. Then
$S_{n,n}\neq 0$.
\end{lemma}
\begin{proof}
We approximate $S$ by a finite sum of monomials $S_\al S_\beta^*$. Then $S_{n,r}$
is approximated by monomials $S_\al S_\beta^*$ with $|\alpha|=r$ and $|\alpha|-|\beta|=n$.

For such monomials, and $m>|\beta|$, we have 
$S_\al S_\beta^*P_{m,m}\phimod\subset P_{m+|\al|-|\beta|, m+|\al|-|\beta|}S_\al S_\beta^*\phimod$,
and by considering $[S_\beta S_\gamma]\in\phimod$ we 
see that $S_\al S_\beta^*P_{m,m}\neq 0$. Hence 
$P_{m+|\al|-|\beta|, m+|\al|-|\beta|}S_\al S_\beta^*P_{m,m}\neq 0$ for $m>|\beta|$.
Hence $(S_\al S_\beta^*)_{|\al|-|\beta|,|\al|-|\beta|}=(S_\al S_\beta^*)_{n,n}\neq 0$
and so also $S_{n,n}\neq 0$.
\end{proof}

\begin{lemma}
\label{lem:finite-sum}
Let $v\in O_E$ (or $\tilde{O_E}$) be a partial isometry 
with range and source projections in $\tilde{A}$. Then 
$\varphi(v)$ is a finite sum of `homogenous' components $v_{m,s}$.
\end{lemma}
\begin{rmk}
We are effectively repeating the argument 
of \cite[Lemmas 4.4 and 4.5]{careyetal11} for modular unitaries and
partial isometries. 
We know that $\ei^{it\D}\varphi(S)\ei^{-it\D}\in \varphi(\Pim)$ 
for partial isometries $S$ with range and source in $A$, and that is all
we will need.
\end{rmk}
\begin{proof}
We first suppose that in fact $v$ is unitary, and define $w_t=\varphi(v^*)\ei^{it\D}\varphi(v)\ei^{-it\D}$.
It follows from Lemma \ref{lem:chop-vee} that $w_t$ commutes (strongly) with $\D$
for all $t$.
Then
\begin{align*}
w_{t+s}&=\varphi(v^*)\ei^{i(t+s)\D}\varphi(v)\ei^{-i(t+s)\D}\\
&=\varphi(v^*)\ei^{it\D}\varphi(v)\ei^{-it\D}\ei^{it\D}\varphi(v^*)\ei^{-it\D}\ei^{i(t+s)\D}\varphi(v)\ei^{-i(t+s)\D}\\
&=w_t\,\ei^{it\D}\big(\varphi(v^*)\ei^{is\D}\varphi(v)\ei^{-is\D}\big)\ei^{-it\D}\\
&=w_t\,\ei^{it\D}w_s\ei^{-it\D}\\
&=w_t\,w_s.
\end{align*}
Hence $w_t$ is a norm continuous path of unitaries in $A$, whence $w_t=\ei^{ita}$ for some
$a=a^*\in A$. Thus $\ei^{it\D}\varphi(v)\ei^{-it\D}=\varphi(v)\ei^{ita}$. Recall now that we can
choose $\D$ to have only integral eigenvalues, and so $\varphi(v)=\varphi(v)\ei^{i2\pi a}$.
Hence $a$ has spectrum a finite subset of $\Z$, and we then easily see that only finitely many
components $v_{m,s}$ can be non-zero.
In the general case we replace $v$ by the unitary
$$
\begin{pmatrix} 1-v^*v & v^*\\ v & 1-vv^*\end{pmatrix}
$$
and argue as above.
\end{proof}

\begin{lemma}
\label{lem:vee-modular}
For any partial isometry $v$ over $\tilde{O_E}$ 
with range and source projections in $\tilde{A}$, the
operators $\varphi(v)\D \varphi(v^*)$ and $\varphi(v^*)\D\varphi(v)$ commute with 
both the kernel projection of $\D$ and the non-negative spectral projection of $\D$, given by $Q$.
\end{lemma}
\begin{proof}
We assume for simplicity that $v\in \tilde{O_E}$.
Using Lemmas \ref{lem:chop-vee} and \ref{lem:finite-sum}, we see that
the following computation is justified and yields the first claim:
\begin{align*}
v\D v^*=\sum_{m,s} vP_{m,s}\psi(m,s)v^*
&=\sum_{m,s,n,r} v_{n,r}P_{m,s}\psi(m,s)v^*
=\sum_{m,s,n,r}P_{m+n,s+r}\psi(m,s)v_{n,r}v^*\\
&=\sum_{m,s,n,r}P_{m+n,s+r}\psi(m,s)v_{n,r}v^*_{n,r}.
\end{align*}
The claims about $v^*\D v$ follow in the same way.
\end{proof}

\section{The isomorphism $K_*(M(A,O_E))\to K_*(A)$ and the $KK$-equivalence}
To prove that $\cdot\otimes_M\hat{\D}:\,K_*(M(A,O_E))\to K_*(A)$ 
gives an isomorphism we need only show 
that $\cdot\otimes_M\hat{\D}$ makes the diagram commute, 
since it then follows from the five lemma that taking 
the Kasparov product with $\hat{\D}$ is an isomorphism. 

To prove that $\cdot\otimes_M\hat{\D}:\,K_*(M(A,O_E))\to K_*(A)$ yields a 
$KK$-equivalence requires much more in general, but follows relatively easily in the boostrap
case. We go further, and provide an explicit inverse when it exists, and conjecture that
in fact it is an inverse in all generality.

\subsection{The isomorphism in $K$-theory}

We now know enough to prove the commutation of the diagram in Equation \eqref{eq:diagram}.
We first consider $\cdot\ox_{M}[\hat\D]:\,K_0(M(A,O_E))\to K_0(A)$, and in this
situation begin by considering $v=v_{m,s}$ `homogenous'.  

By \cite[Lemma 2.14]{kappa}, the range of $Q$ is the range of $\sum_{n\geq 0}P_{n,n}$, hence $QvQ=0$ unless $s=m$. 
For $s=m$ we have
\begin{align*}
&{\rm Index}(QvQ:v^*vF_E\to vv^*F_E)\ox_A({\rm Id}_{KK(A,A)}-[E])\\
&=\begin{cases}
\left([\oplus_{j=0}^{-m-1}v^*vE^{\ox j}]\right)\ox_A([{\rm Id}_{KK(A,A)}]-[E]) & m<0 \\
\left(-[\oplus_{j=0}^{m-1}vv^*E^{\ox j}]\right)\ox_A([{\rm Id}_{KK(A,A)}]-[E]) & m\geq0 \\
\end{cases} \\
&=\left\{\begin{array}{ll} {[v^*vA]} - {[v^*vE^{\ox -m}]}  & m<0\\  
{-[vv^*A]} +{[vv^*E^{\ox m}]}  & m\geq0\end{array}\right.
\end{align*}
where the last equality follows from a telescopic argument.
So to prove that
$$
{\rm Index}(QvQ:v^*vF_E\to vv^*F_E)\ox_A({\rm Id}_{KK(A,A)}-[E])
=\ev_*([v])=[v^*v]-[vv^*]
$$
we are reduced to proving the isomorphisms of $A$-modules
$$
\begin{aligned}[]
vv^*E^{\otimes m} &\simeq v^*vA  &\mbox{for }m>0\quad\mbox{and}\quad  
v^*vE^{\otimes |m|} & \simeq vv^*A &\mbox{for }m<0.
\end{aligned}
$$
This is straightforward though, by the following argument.

For $m<0$, the map $v:v^*vE^{\otimes|m|}\to vE^{\otimes|m|}\subset A$ is a one-to-one
$A$-module map, which is onto its image, which is contained in $vv^*A$. Hence
$v^*vE^{\otimes|m|}$ and $vE^{\otimes|m|}$ are isomorphic.

For $m>0$, the map $v^*:vv^*E^{\otimes|m|}\to v^*E^{\otimes |m|}\subset A$ is a one-to-one
$A$-module map, which is onto its image, which is contained in $v^*vA$. Hence
$vv^*E^{\otimes|m|}$ and $v^*E^{\otimes|m|}$ are isomorphic.

Thus the result is true for homogenous partial isometries, and likewise for 
direct sums of homogenous partial isometries, and by Lemma \ref{lem:finite-sum} this
is enough.
This gives commutativity of the diagram 
\[
\xymatrix{ K_0(M) \ar[rr]^{\ev_{*}} \ar[d]^{\cdot\ox[\hat\D]} &\ \ & K_0(A) \ar[d]^{=} \\
K_0(A) \ar[rr]_{\cdot\ox ({\rm Id}_A-[E])} &\ \ & K_0(A)}
\]
and hence an isomorphism $\cdot\ox_M[\hat\D]:K_0(M(A,O_E))\to K_0(A)$. 
To complete the argument, we need to consider suspensions. 

If $f\in \S M(A,O_E)$ we let $f(t)=g_t$ with $g_t\in M(A,O_E)$ for all
$t\in\R$. Then define
$$
\Psi:\S M(A,O_E)\to M(\S A, \S O_E),\quad (\Psi(f)(s))(t)=g_t(s),\ \ s\in[0,\infty),\ \ t\in\R,
$$
and check that $\Psi$ is an isomorphism. Hence,
in particular, $K_1(M(A,\O_E))\cong K_0(M(\S A,\S O_E))$.

Next we observe that $O_{\S E}\cong \S O_E$. The isomorphism is defined
on generators by $\varphi(S_{f\ox e})=f\ox S_e$, and using the 
gauge invariant uniqueness theorem, as in \cite[Theorem 6.4]{Katsura}, we see that the map is injective, 
and then since the range contains the generators of $\S O_E$, it is an isomorphism.

The unitary isomorphism 
of Kasparov modules $(\S A,\S E_{\S A},0)=(C_0(\R),C_0(\R)_{C_0(\R)},0)\ox_\C(A,E_A,0)$
shows that the suspension of the map $\cdot\ox_A (\,(A,A_A,0)-(A,E_A,0)\,)$ is the map
$\cdot\ox_{\S A}(\,(\S A,\S A_{\S A},0)-(\S A,\S E_{\S A},0)\,)$. A similar but
easier statement holds for the suspension of the evaluation map, and so combining these
various facts we
find that the diagram 
\[
\xymatrix{ K_1( M(A,O_E)) \ar[rr]^{\ \ \ \ \ \ \S\ev_{*}} \ar[d]^{\cdot \ox[\hat\D]} & \ \ &K_1(A) \ar[d]^{=} \\
K_1(A) \ar[rr]_{\S \ox ({\rm Id}_A-[E])} &\ \ & K_1(A)}
\]
is given by
\[
\xymatrix{ K_0(M(\S A,\S O_E)) \ar[d]^{\cdot \ox[\hat\D]}  \ar[rr]^{\ \ \ \ \ \ \ev_{*}}&\ \ & K_0(\S A) \ar[d]^{=} \\
K_0(\S A) \ar[rr]_{\ox ({\rm Id}_{\S A}-[\S E])} &\ \ & K_0(\S A)},
\]
where now ${\rm ev}_*$ is the evaluation map corresponding to the inclusion
of $\S A$ into $\S O_E$.
Now by Propositions \ref{prop:suspend} and \ref{prop:suspend-assumps}, 
$\S E$ satisfies all the assumptions that $E$ does. Thus our proof that the `even part' 
of the diagram commutes now holds verbatim to show that the `odd part' of the diagram commutes.

\subsection{The $KK$-equivalence and the main theorem}
\label{sec:KK-equiv}

We conclude by showing that the class of $[\hat{\D}]$ 
not only implements an isomorphism in $K$-theory, but an actual 
$KK$-equivalence when $A$ is in the bootstrap class. 

First observe that by {\cite[Proposition 23.10.1]{Blackadar}}, 
if $A,\, B$ are two \Cs algebras in the bootstrap class, then $\alpha \in KK(A,B)$ 
is a $KK$-equivalence if and only if the induced map 
$\cdot\otimes_A \alpha : K_*(A) \to K_*(B)$ is invertible. 
This follows from the Universal Coefficient Theorem of \cite{RoSch}. 

Next, whenever the coefficient algebra $A$ of the correspondence $E_A$ 
belongs to the bootstrap class, 
so does the algebra $\Pim$ (cf. \cite[Proposition 8.8]{Katsura}), and we obtain 
a $KK$-equivalence in this case.

Hence, provided that the coefficient algebra is contained in the bootstrap class,
the class $[\hat\D]\in KK(O_E,A)$ is a $KK$-equivalence. The problem with this
approach is two-fold. On the one hand we need to know or verify that the coefficient algebra
is in the bootstrap class. While this is often possible, knowing that $\hat\D$ is a $KK$-equivalence
does not provide a representative of the other half of the equivalence.

We ameliorate both these problems by providing an explicit representative 
for the other half of the $KK$-equivalence, when it is one.

In our particular situation, we can choose a 
countable frame $\{e_i\}_{i \geq 1}$ for the right module
$E_A$, and define the (possibly infinite) matrix over $\widetilde{O_E}$ by
$$
w=
\begin{pmatrix}S_{e_1}^* & 0 & \cdots & 0 & \\ S_{e_2}^* &  0 & \cdots & 0\\ 
\vdots & \vdots & \ddots & \vdots \end{pmatrix}.
$$
Then $w^*w={\rm Id}_{\widetilde{O_E}} \oplus 0_{\infty}
=\iota_{A,O_E}({\rm Id}_{\widetilde{A}}) \oplus 0_{\infty}$ and 
$$
ww^*=\begin{pmatrix} (e_1|e_1)_A &(e_1|e_2)_A & \cdots \\
(e_2|e_1)_A  & (e_2|e_2)_A & \cdots\\
\vdots & \ddots & \vdots
\end{pmatrix}= (e_i|e_j)_{i,j \geq 1} = : p_E\in M_{\infty}(A).
$$
When the left action of $A$ on $E$ is injective and $E$ has finite right Watatani index, 
the projection $p_E$ lies in $\widetilde{\iota_{A,O_E}(A)}$, 
and so $w$ defines a class in $K_0(M)$.
We will prove this fact below.

We can explicitly realise $[w]$ as a difference of classes of projections 
over $\tilde{M}(A,O_E)$.\footnote{Here  $\tilde{M}(A,O_E)$ is the minimal unitisation of $M(A,O_E)$.
As usual, the equality of the classes of $p_w(\infty)$ and $1_N$ gives us classes
in the $KK$ groups for $M(A,O_E)$. See \cite[Corollary 1, Section 7]{KasparovTech}} 
Making the identification $[w]=[p_w]-[1_N]$, we have
$$
p_w(t)=
\begin{pmatrix} 1_N-\frac{1}{1+t^2}p_E & \frac{-it}{1+t^2}w\\ 
\frac{it}{1+t^2}w^* & \frac{1}{1+t^2}{\rm Id}_{O_E}\end{pmatrix}
=
\begin{pmatrix} \frac{1}{1+t^2}(1_N-p_E)+\frac{t^2}{1+t^2}1_N & \frac{-it}{1+t^2}w\\ 
\frac{it}{1+t^2}w^* & \frac{1}{1+t^2}{\rm Id}_{O_E}\end{pmatrix}.
$$
The frame $\{e_i\}_{i \geq 1}$ gives a stabilisation map 
$\psi:\,E\to H_A=H\ox A$ (for any separable Hilbert space $H$) 
by defining $\psi(e)=((e_j|e)_A)_{j\geq 1}$.
Since $E_A$ carries a left action of $A$, say $\phi:A\to \End_A(E)$, 
so too does $p_EH_A$ with $\psi\circ\phi(a)\circ\psi^{-1}=((e_i|\phi(a)e_j)_A)_{i,j}$. 

There are two important features of the representation of $A$ on $p_EH_A$. The first
is that $w$ defines a class in $K_0(M(A,O_E))$. Let $(u_n)_{n\geq 1}$ be an approximate
identity for $A$, and recall that $A$ acts injectively and by compacts to see 
that $((e_i|\phi(u_n)e_j)_A)_{i,j}$ converges strictly to 
$p_E$. Extending the representation of $A$ to $\widetilde{A}$, we see that
$p_E=\psi\circ\phi(a)\circ\psi^{-1}({\rm Id}_{\widetilde{A}})$. Thus $p_E$ is a class over
$\widetilde{A}$.

The second important feature is the ability to
inflate from $K$-theory classes to $KK$-classes.
Since $w(0_{N-1}\oplus\phi(a))w^*=\psi\circ\phi(a)\circ\psi^{-1}$ and
$w^*\psi\circ\phi(a)\circ\psi^{-1}w=0_{N-1}\oplus\phi(a)$, it is straightforward
to check that for all $t\in [0,\infty)$
$$
p_w(t)\begin{pmatrix} ((e_i|\phi(a)e_j)_A)_{i,j} & 0\\ 0 & \phi(a)\end{pmatrix}
=\begin{pmatrix} ((e_i|\phi(a)e_j)_A)_{i,j} & 0\\ 0 & \phi(a)\end{pmatrix}p_w(t)
$$
as operators on $O_E^{2N}$ (or $A^{2N}$ for $t=0$). 
Hence we can enrich the $K$-theory class
$$
[w]=\left[\left(\C,\begin{pmatrix} p_w\tilde{M}(A,O_E)^{2N}\\ \tilde{M}(A,O_E)^N\end{pmatrix}, 0\right)\right]
\in KK(\C,M(A,O_E))
$$
to a class
$$
[W]=\left[\left(A,\begin{pmatrix} p_w\tilde{M}(A,O_E)^{2N}\\ \tilde{M}(A,O_E)^N\end{pmatrix}, 0\right)\right]
\in KK(A,M(A,O_E)).
$$

These two classes are related, via the natural inclusion $\iota_{\C,A}: \C \to A$, by
$
[w]=\iota_{\C,A}^*\left([W]\right).
$
\begin{lemma}
\label{lem:pairing1}
Let $[\D]\in KK^1(O_E,A)=KK(\S O_E,A)$ be the class of the defining extension for $O_E$,
$(M(A,O_E),\hat\Xi_A,\hat\D)$ the lift to the mapping cone, and
$[W]\in KK(A,M(A,O_E))$ the class defined above. Then
$$
[W]\ox_{M(A,O_E)}[\widehat\D]=-\Id_{KK(A,A)}.
$$
\end{lemma}
\begin{proof}
Applying \cite[Theorem 5.1]{CPR1} gives 
\begin{align*}
[w]\ox_{M(A,O_E)}[\widehat{\D}]
=-{\rm Index}\big(P\ox1_NwP\ox 1_N:w^*w(\Xi)^N\to ww^*(\Xi)^N\big)
\end{align*}
where $P$ is the non-negative spectral projection of $\D$. 
Since the 
non-negative spectral projection of $\D$ 
is the projection onto a copy of the Fock space, we have 
$$
\ker(P\ox1_NwP\ox1_N)=A_A=E^{\ox 0}_A,\qquad \ker(P\ox1_Nw^*P\ox1_N)=\{0\}.
$$
We can interpret 
the index not just as a difference of right $A$-modules, but as a difference
of $A$-bimodules. This works because the 
left action of $A$ commutes with $\D$ and so $P$. Hence
$$
[W]\ox_{M}[\widehat{\D}]=-[(A,A_A,0)]=-\Id_{KK(A,A)}
$$
as was to be shown.
\end{proof}
From Lemma \ref{lem:pairing1}, we know that 
$-[\widehat\D]\ox_A[W]\in KK(M(A,O_E),M(A,O_E))$
is an idempotent element. {\em In particular, $[\widehat{\D}]\ox_A\cdot$ is always injective and
$[W]\ox_M\cdot$ is always surjective and injective on the image of $[\widehat{\D}]\ox_A\cdot$.
Thus as soon as $[\widehat{\D}]\ox_A\cdot$ is surjective, $[\widehat{\D}]$ is a $KK$-equivalence. Similarly, the map $\cdot\ox_A[W]$ is always injective and surjective on the range 
of $\cdot\ox_M[\hat\D]$.}

One approach to showing that $[W]$ is in fact an inverse for $[\widehat\D]$
would be to show that the diagram
\begin{align}\xymatrix{\cdots \ar[r]^{\iota_{*}}&K_0(\Pim) \ar[r]^-{j_*^B} 
&K_{1}(M) \ar[r]^-{\ev_*} &
K_{1}(A) \ar[r]^{\iota_{*}}& K_1(\Pim) \ar[r]^-{j_*^B}
&K_{0}(M) \ar[r]^-{\ev_*}&
K_{0}(A) \ar[r]^{\iota_{*}}&\cdots \\
\cdots \ar[r]^{\iota_{*}}&K_0(\Pim) \ar[u]^{=}\ar[r]^-{\partial} &K_{1}(A) \ar[u]^{\cdot \ox [W]}\ar[r]^-{1-[E]} &
K_{1}(A) \ar[u]^{=}\ar[r]^{\iota_{*}}&K_1(\Pim) \ar[u]^{=} \ar[r]^-{\partial} &K_{0}(A)  \ar[u]^{\cdot \ox [W]}\ar[r]^-{1-[E]} &
K_{0}(A) \ar[u]^{=}\ar[r]^{\iota_{*}}&\cdots }
\label{eq:diagramW}
\end{align}
commutes. The composition $[W]\ox_{M(A,O_E)}[{\rm ev}]$ is just the module
$$
\left[\left(A,\begin{pmatrix} (1_N-p_E)\tilde{A}^{N}\\ A \\ 
\tilde{A}^N\end{pmatrix}, 0\right)\right]
$$
with grading $(1_N-p_E)\oplus 1 \oplus -1_N$. So 
$[W]\ox_{M(A,O_E)}[{\rm ev}]=[A]-[p_EA^N]=[A]-[E]$.
Thus $-[W]$ makes one square commute, and we could try
to show that the square to the left of $W$ commutes (up to sign) as well.

This means showing that $-[\D]\ox_A[W]=[j]\in KK(\S O_E,M(A,O_E))$, which is implied by 
the stronger condition $-[\widehat\D]\ox_A[W]={\rm Id}_{KK(M,M)}\in KK(M(A,O_E),M(A,O_E))$.
We have not been able to prove this equality in general, and leave it as an open problem.

%
%
%
%
%
If $A$ is in the bootstrap class, then so too is $O_E$ and so $M(A,O_E)$.
In this case $\cdot\ox_M[\hat\D]$ is an isomorphism, hence the map $\cdot\ox_A[W]$, which is always injective on the range 
of $\cdot\ox_M[\hat\D]$, is an isomorphism as well: the inverse of $\cdot\otimes_M[\hat\D]$.

\begin{thm}
\label{thm:KK-equivalence}
Let $E$ be a bi-Hilbertian $A$-bimodule of finite right Watatani index, full as a right module
with injective left action, and satisfying Assumptions \ref{ass:one} and \ref{ass:two} 
on pages 8 and 9 respectively. Let
$(O_E,\phimod,\D)$ be the unbounded representative of the defining extension of $O_E$,
and $(M(A,O_E),\hat\Xi_A,\hat\D)$ the lift to the mapping cone. Then
$$
\cdot\ox_{M(A,O_E)}[(M(A,O_E),\hat\Xi_A,\hat\D)]:\,K_*(M(A,O_E))\to K_*(A)
$$
is an isomorphism that makes diagrams in $K$-theory commute. If furthermore the 
algebra $A$ belongs to the bootstrap class, the Kasparov product with the class
$[(M(A,O_E),\hat\Xi_A,\hat\D)] \in KK(M(A,O_E),A)$ is a $KK$-equivalence.
%
%
Together with the identity map, 
$\cdot\ox_{M(A,O_E)}[(M(A,O_E),\hat\Xi_A,\hat\D)]$ induces
an isomorphism of $KK$-theory exact sequences. 
\end{thm}

Applying the result to the graph $C^*$-algebra of a locally finite directed graph
with no sources nor sinks yields a well-known exact sequence for computing the $K$-theory.
With $G=(G^0,G^1,r,s)$ the directed graph ($G^0=$vertices, $G^1$=edges),
$A=C_0(G^0)$ and $E=C_0(G^1)$ we have $O_E=C^*(G)$ and we have the exact sequence.
$$
\xymatrix{ 0\ar[r] & K_1(C^*(G))\ar[r]& K_0(M(A,C^*(G)))\ar[r]^-{\ev_*}& K_0(A)\ar[r] &K_0(C^*(G)) \ar[r]& 0}.
$$
Using $K_0(A)=\oplus^{\vert G^0 \vert }\Z$, together with the isomorphism $K_0(M(A,C^*(G)))\cong K_0(A)$ given
by $\hat\D$, gives
$$
\xymatrix{ 0\ar[r] & K_1(C^*(G))\ar[r]& \bigoplus^{\vert G^0 \vert }\Z\ar[r]^-{1-V^T}& 
\bigoplus^{\vert G^0 \vert }\Z\ar[r] &K_0(C^*(G)) \ar[r]& 0}.
$$
where $V$ is the vertex matrix of the graph $G$, given by 
$$
V(i,j) := \{e \in G^1 :\, s(e)=v_i,\  r(e)=v_j\}.
$$
Similarly, since $A=C_0(G^0)$ is in the bootstrap class, in $K$-homology we find
$$
\xymatrix{ 0\ar[r] & K^0(C^*(G))\ar[r]&\prod^{\vert G^0 \vert }\Z\ar[r]^-{1-V}& 
\prod^{\vert G^0 \vert }\Z\ar[r] &K^1(C^*(G)) \ar[r]& 0},
$$
and these results recapture the results of \cite{DT} for non-singular graphs.

\end{document}